\documentclass[12pt]{amsart}
\usepackage{amscd,amsmath,amsthm,amssymb}
\usepackage{pstcol,pst-plot,pst-3d}%\usepackage[T1]{fontenc}
\usepackage{lmodern,pst-node}%\usepackage{geometry}
\usepackage{lineno}
\def\NZQ{\Bbb}               % the font for N,Z,Q,R,C
\def\NN{{\NZQ N}}

\def\ZZ{{\NZQ Z}}
\def\RR{{\NZQ R}}

\def\frk{\frak}               % font for "Fraktur"

\def\pp{{\frk p}}

\def\qq{{\frk q}}

\def\mm{{\frk m}}

\def\Phi{{\frk n}}
\def\Phi{{\frk N}}
\def\MP{{\mathcal P}}

\def\MT{{\mathcal T}}

\def\ab{{\bold a}}

\def\xb{{\bold x}}

\def\cb{{\bold c}}
\def\ub{{\bold u}}
\def\opn#1#2{\def#1{\operatorname{#2}}} % to make operators
\opn\chara{char} \opn\length{\ell} \opn\pd{pd} \opn\rk{rk}
\opn\projdim{proj\,dim} \opn\injdim{inj\,dim} \opn\rank{rank}
\opn\depth{depth} \opn\grade{grade} \opn\height{height}
\opn\embdim{emb\,dim} \opn\codim{codim}

\opn\Tr{Tr} \opn\bigrank{big\,rank}
\opn\superheight{superheight}\opn\lcm{lcm}
\opn\trdeg{tr\,deg}%\emph{
\opn\reg{reg} \opn\lreg{lreg} \opn\ini{in} \opn\lpd{lpd}
\opn\size{size}\opn\bigsize{bigsize}
\opn\cosize{cosize}\opn\bigcosize{bigcosize}
\opn\sdepth{sdepth}\opn\sreg{sreg}
\opn\link{link}\opn\fdepth{fdepth}
\opn\div{div} \opn\Div{Div} \opn\cl{cl} \opn\Cl{Cl}
\opn\Spec{Spec} \opn\Supp{Supp} \opn\supp{supp} \opn\Sing{Sing}
\opn\Ass{Ass} \opn\Min{Min}\opn\Mon{Mon} \opn\dstab{dstab} \opn\astab{astab}
\opn\Ann{Ann} \opn\Rad{Rad} \opn\Soc{Soc}
\opn\Im{Im} \opn\Ker{Ker} \opn\Coker{Coker} \opn\Am{Am}
\opn\Hom{Hom} \opn\Tor{Tor} \opn\Ext{Ext} \opn\End{End}
\opn\Aut{Aut} \opn\id{id}

\opn\nat{nat}
\opn\pff{pf}%   \pf exists already
\opn\Pf{Pf} \opn\GL{GL} \opn\SL{SL} \opn\mod{mod} \opn\ord{ord}
\opn\Gin{Gin} \opn\Hilb{Hilb}\opn\sort{sort}
\opn\aff{aff} \opn\con{conv} \opn\relint{relint} \opn\st{st}
\opn\lk{lk} \opn\cn{cn} \opn\core{core} \opn\vol{vol}
\opn\link{link} \opn\star{star}\opn\lex{lex} \opn\sat{sat}
\opn\gr{gr}

\def\pot#1#2{#1[\kern-0.28ex[#2]\kern-0.28ex]}

\opn\dirlim{\underrightarrow{\lim}}
\opn\inivlim{\underleftarrow{\lim}}
\let\union=\cup
\let\sect=\cap

\let\Union=\bigcup
\let\Sect=\bigcap

\let\to=\rightarrow
\let\To=\longrightarrow
\let\wt=\widetilde{}
\def\Implies{\ifmmode\Longrightarrow \else
        \unskip${}\Longrightarrow{}$\ignorespaces\fi}
\def\implies{\ifmmode\Rightarrow \else
        \unskip${}\Rightarrow{}$\ignorespaces\fi}
\def\iff{\ifmmode\Longleftrightarrow \else
        \unskip${}\Longleftrightarrow{}$\ignorespaces\fi}

\let\:=\colon
\newtheorem{Theorem}{Theorem}[section]
\newtheorem{Lemma}[Theorem]{Lemma}
\newtheorem{Corollary}[Theorem]{Corollary}
\newtheorem{Proposition}[Theorem]{Proposition}
\newtheorem{Remark}[Theorem]{Remark}

\newtheorem{Example}[Theorem]{Example}
\newtheorem{Examples}[Theorem]{Examples}

\let\epsilon\varepsilon
\let\kappa=\varkappa
\def\qed{\ifhmode\textqed\fi
      \ifmmode\ifinner\quad\qedsymbol\else\dispqed\fi\fi}
\def\textqed{\unskip\nobreak\penalty50
       \hskip2em\hbox{}\nobreak\hfil\qedsymbol
       \parfillskip=0pt \finalhyphendemerits=0}
\def\dispqed{\rlap{\qquad\qedsymbol}}

\opn\dis{dis}
\def\pnt{{\raise0.5mm\hbox{\large\bf.}}}

\opn\Lex{Lex}

\begin{document}
%\linenumbers
\title {Monomial ideals with primary components given by powers of monomial prime ideals}
\author {J\"urgen Herzog and Marius Vladoiu}\thanks{The second author was partially supported from grant PN--II--ID--PCE--2011--3--1023 nr. 247/2011 awarded by UEFISCDI}
\address{J\"urgen Herzog, Fachbereich Mathematik, Universit\"at Duisburg-Essen, Campus Essen, 45117
Essen, Germany} \email{juergen.herzog@uni-essen.de}

\address{Marius Vladoiu, Faculty of Mathematics and Computer Science, University of Bucharest, Str.
 Academiei 14, Bucharest, RO-010014, Romania, and}
\address{Simion Stoilow Institute of Mathematics of Romanian Academy, Research group of the project ID--PCE--2011--3--1023, P.O.Box 1--764, Bucharest
014700, Romania}\email{vladoiu@gta.math.unibuc.ro}

\begin{abstract}
We characterize monomial ideals which are intersections of monomial prime ideals and study classes of ideals with this property, among them polymatroidal ideals.
\end{abstract}
\maketitle

\section*{Introduction}

In this paper we study monomial ideals
which are intersections of  powers of monomial prime ideals, and call them  {\em monomial ideals of intersection type}. Any squarefree monomial ideal is of intersection type since it is actually an intersection of monomial prime ideals. Obviously, among the non-radical monomial ideals, the monomial ideals of intersection type are closest to squarefree monomial ideals.  Hence one may expect that they are easier to understand than general monomial ideals. Indeed, various special types of such ideals have been studied in the literature, among them the so-called ideals of tetrahedral curves, see \cite{MN} and \cite{FMN}. In \cite{FV} the authors discuss the problem when a monomial ideal of intersection type is componentwise linear. In that paper monomial ideals of intersection type are called intersections of Veronese type ideals.

In Section 1 of the present paper we deal with the problem to characterize those monomial ideals which are of intersection type. The answer is given in Theorem~\ref{prime_intersection} where it is shown that a monomial ideal $I\subset S=K[x_1,\ldots,x_n]$ is of intersection type if and only if for each associated prime ideal $\pp$ of $I$ the minimal degree of a generator of the monomial localization $I(\pp)$ of $I$ is bigger than the maximal degree of a non-zero socle element of $S(\pp)/I(\pp)$. Moreover, it is shown that if $I$ is of intersection type,  its presentation as an intersection of powers of monomial prime ideals is uniquely determined. We call this presentation the {\em canonical primary decomposition} of $I$.  One should notice that an ideal of intersection type, if it has embedded prime ideals, may have primary decompositions, different from the canonical one. The exponents of the powers of the monomial prime ideals appearing in the intersection are bounded above. Indeed, if $I=\Sect_{\pp\in \Ass(S/I)}\pp^{d_\pp}$, then $d_\pp\leq \reg(I(\pp))$. We say that $I$ is of {\em strong intersection type}, if $d_\pp=\reg(I(\pp))$ for all $\pp\in \Ass(S/I)$, and show that this is the case if and only if $I(\pp)$ has a linear resolution for all $\pp\in \Ass(S/I)$. It is clear that any squarefree monomial ideal is of strong intersection type, while monomial ideals of intersection type with embedded prime ideals may or may not be of strong intersection type.

In Section 2 we consider classes of ideals which are of (strong) intersection type. It is shown in Proposition~\ref{polym_c} that any polymatroidal ideal is of strong intersection type, and in Theorem~\ref{primdec_polym} we show that the canonical primary decomposition of polymatroidal ideals  is given in terms of the rank function of the underlying polymatroid.  Unfortunately we do not have a uniform description of the associated prime ideals  in terms of the rank function.  However for polymatroidal ideals of Veronese type and for transversal polymatroidal ideals the canonical primary decomposition is made explicit, see Examples~\ref{known_polym}. On the other hand, a coherent primary decomposition of  any polymatroidal ideal, which however may not always be irredundant, is given in Theorem~\ref{redundant}. Among the ideals of Borel type, the principal Borel ideals are precisely those which are of strong intersection type, as shown in Proposition~\ref{borel_c}. In the remaining part of Section~2 we consider powers of edge ideals. In Theorem~\ref{depth_2power} we characterize,  terms of their underlying graph, those edge ideals   whose second power is of intersection type, and prove in Corollary~\ref{higher} that if the square of an edge ideal is not of intersection type, then so are all its higher powers.  It would be of interest to classify all graphs with the property that all its powers are of intersection type.

In the last section we list a few properties of ideals of intersection type. One nice property is that any ideal of intersection type is integrally closed, see Proposition~\ref{integrally_closed}.  In Corollary~\ref{supporting} the supporting hyperplanes of the Newton polyhedron of an ideal of intersection type are given in terms of its canonical primary decomposition. Finally in Theorem~\ref{symbolic_power} it is shown that if $I^k$ is of intersection type, then its  $\ell^{th}$ symbolic power $I^{(\ell)}$  of $I$ is contained in $I^k$, where $\ell = \reg(I^k)$. By using this fact and a result of Kodiyalam \cite{K}, and assuming that all powers of $I$ are of intersection type, we show in the second part of Theorem~\ref{symbolic_power} that $I^{(rk)}\subset I^k$ for all $k\gg 0$, where $r=\rho(I)+1$ and where  $\rho(I)$ is the minimum of $\theta(J)$  over all graded
reductions $J$ of $I$. Here $\theta(J)$ denotes the maximal degree of a generator in a minimal set of generators of $J$.

\section{When is a monomial ideal the intersection of  powers of monomial prime ideals? }

Let $S=K[x_1,\ldots,x_n]$ be the polynomial ring in $n$ variables over the field $K$ with graded maximal ideal $\mm=(x_1,\ldots,x_n)$, and let $I\subset S$ be a monomial ideal. In general, unless $I$ is squarefree,  the ideal $I$ can not be written as intersection of powers of monomial prime ideals. In the following we give a necessary and sufficient condition for $I$  to have such a presentation, and show that this presentation is unique once it is irredundant.
To describe the main result of this section we introduce some notation. Let $\pp$ be a monomial prime ideal, that is, $\pp$ is a prime ideal generated by a subset of the variables. As in \cite{HRV} we denote by $I(\pp)$ the monomial localization of $I$,  and by $S(\pp)$ the polynomial ring over $K$ in the variables which belong to $\pp$. Recall that $I(\pp)\subset S(\pp)$ is the monomial ideal which is obtained from $I$ as the image of the $K$-algebra homomorphism $\varphi\: S\to S(\pp)$ with $\varphi(x_i)=x_i$ if $x_i\in \pp$, and $\varphi(x_i)=1$,  otherwise. Monomial localizations are compatible with products and intersections. In other words, if $I$ and $J$ are monomial ideals, then $(IJ)(\pp)=I(\pp)J(\pp)$ and $(I\sect J)(\pp)=I(\pp)\sect J(\pp)$.

We denote by $V^*(I)$ the set of monomial prime ideals containing $I$. It is known that $\Ass(S/I) \subset V^*(I)$ (see for example \cite[Corollary 1.3.9]{HH}),  and that $\pp\in V^*(I)$ belongs to $\Ass(S/I)$ if and only if $\depth(S(\pp)/I(\pp))=0$. The latter is the case if and only if $I(\pp):\mm_\pp\neq I(\pp)$ where $\mm_\pp$ denotes the graded maximal ideal of $S(\pp)$. The $K$-vector space $(I(\pp):\mm_\pp)/ I(\pp)\subset S(\pp)/I(\pp)$ is called the {\em socle} of $S(\pp)/I(\pp)$, and denoted by $\Soc(S(\pp)/I(\pp))$. Thus $\pp\in \Ass(S/I)$ if and only if $\Soc(S(\pp)/I(\pp))\neq 0$. Since the colon ideal of monomial ideals is again a monomial ideal it follows that $\Soc(S(\pp)/I(\pp))$ is generated by the elements $u+I(\pp)$ where $u\in S(\pp)$ is a monomial with $u\not\in I(\pp)$ and  $\mm_\pp u\subset I(\pp)$. Since $\Soc(S(\pp)/I(\pp))$ is a finite dimensional graded $K$-vector space we may define the number
\[
\max(\Soc(S(\pp)/I(\pp)))=\max\{i\:\, (\Soc(S(\pp)/I(\pp)))_i\neq 0\}
\]
for any $\pp\in \Ass(S/I)$. We also set
\[
\min(M)=\min\{i\:\; M_i\neq 0\}
\]
for any finitely generated  graded $S(\pp)$-module $M$.  Finally we denote by $\wt{I(\pp)}$ the {\em saturation} of $I(\pp)$ which is defined to be the monomial ideal $\Union_k I(\pp): \mm_\pp^k$.

In the sequel we will also use the following notation: for a graded ideal $I$  and an integer $k\geq 0$ we denote by $I_{\geq k}$ the ideal generated by all graded components $I_i$ of $I$ with $i\geq k$.

\begin{Theorem}
\label{prime_intersection}
Let $I$ be a monomial ideal.
\begin{enumerate}
\item[(a)] The following conditions are equivalent:
\begin{enumerate}
\item[(i)] $I$ is an intersection of powers of monomial prime ideals;
\item[(ii)] for all $\pp\in\Ass(S/I)$ there exist positive integers $a_{\pp}$ such that $$I(\pp)=\wt{I(\pp)}\cap\mm_{\pp}^{a_{\pp}};$$
\item[(iii)] $\min(I(\pp))>\max(\Soc(S(\pp)/I(\pp)))$ for all $\pp\in\Ass(S/I)$.
\end{enumerate}
\item[(b)] Let $I=\Sect_{i=1}^r \pp_i^{d_i}$ be an irredundant presentation of $I$ as an intersection of powers of monomial prime ideals. Then  $\Ass(S/I)=\{\pp_1,\ldots,\pp_r\}$ and $$d_i=\max(\Soc(S(\pp_i)/I(\pp_i)))+1 \text{ for all } i=1,\ldots,r.$$
\end{enumerate}
\end{Theorem}
\begin{proof}
(a)  For the proof of (i) \implies (ii) let $I=\Sect_{i=1}^r \pp_i^{a_i}$. We may assume that this presentation is irredundant. Then $\Ass(S/I)=\{\pp_1,\ldots,\pp_r\}$. Fix some integer $j$. Then
\[
I(\pp_j)=\Sect_{i \atop \pp_i\subsetneq\pp_j}\pp_i(\pp_j)^{a_i}\cap\mm_{\pp_j}^{a_j}.
\]
It follows from this presentation that $\wt{I(\pp)}=\Sect_{i,\; \pp_i\subsetneq\pp_j}\pp_i(\pp_j)^{a_i}$.

The proof of (ii)\implies (iii) follows immediately once that it is noticed that (ii) is equivalent to saying that $I(\pp)=\wt{I(\pp)}_{\geq a_{\pp}}$ for all $\pp\in\Ass(I)$.  This identity implies that $\min(I(\pp))\geq a_{\pp}$,  and that the maximal degree of a non-zero element  $u+I(\pp)\in \Soc(S(\pp)/I(\pp))$, $u$ a monomial,  is at most $a_{\pp}-1$ since $u\in\wt{I(\pp)}\setminus I(\pp)$.

%In order to prove that (iii)$\Rightarrow$(ii) we set  $k=\max(\Soc(S(\pp)/I(\pp)))$. Then  $\wt{I(\pp)}_{\geq k+1}\subseteq I(\pp)$. On the other hand, since $\min(I(\pp))\geq k+1$ and $I(\pp)\subset \wt{I(\pp)}$ we also obtain $I(\pp)\subseteq \wt{I(\pp)}_{\geq k+1}$. Therefore $\wt{I(\pp)}_{\geq k+1}=I(\pp)$, as desired.

(iii)\implies (i): For all $\pp\in \Ass(S/I)$ let $d_\pp=\max(\Soc(S(\pp)/I(\pp)))+1$. We claim that $I=\Sect_{\pp\in \Ass(S/I)}\pp^{d_\pp}$. In order to prove this, we first show that $I\subset\Sect_{\pp\in \Ass(S/I)}\pp^{d_\pp}$. Indeed, let $\pp\in \Ass(S/I)$ and suppose that $I\not\subset \pp^{d_\pp}$. Then $I(\pp)\neq I(\pp)_{\geq d_p}$. It follows that $\min(I(\pp))< d_\pp$, a contradiction.

Next assume that $I$ is properly contained in $J=\Sect_{\pp\in \Ass(S/I)}\pp^{d_\pp}$. Then there exists $\pp\in\Min(J/I)\subset \Ass(S/I)$ such that $J(\pp)/I(\pp)=(J/I)(\pp)\neq 0$. Let $u\in J(\pp)\setminus I(\pp)$. Then $\deg u\geq d_\pp$ because $u\in\mm_\pp$. Since $\pp\in\Min(J/I)$, it follows that $J(\pp)/I(\pp)$ has finite length. Thus $\max(\Soc(S(\pp)/I(\pp)))\geq \deg u\geq d_\pp$, a contradiction.

\medskip
(b) Since this presentation is irredundant, it represents an irredundant primary decomposition of $I$. This shows that $\Ass(S/I)=\{\pp_1,\ldots,\pp_r\}$. Fix an integer $i$. Then $I(\pp_i)=\wt{I(\pp_i)}\cap\mm_{\pp_i}^{d_i}=\wt{I(\pp_i)}_{\geq d_i}$ and $I(\pp_i)$ is properly contained in $\wt{I(\pp_i)}$. Thus
\[
\max(\Soc(S(\pp_i)/I(\pp_i)))=\max(\Soc(\wt{I(\pp_i)}/I(\pp_i)))= \max(\Soc(\wt{I(\pp_i)}/\wt{I(\pp_i)}_{\geq d_i}))
\]
It follows that  $\max(\Soc(S(\pp_i)/I(\pp_i)))=d_i-1$, as desired.
\end{proof}

We call a monomial ideal satisfying the equivalent conditions of Theorem~\ref{prime_intersection} to be a monomial ideal of {\em intersection type}.
The presentation of an ideal of intersection type as given in Theorem~\ref{prime_intersection}(b) is called the {\em canonical primary decomposition} of $I$. Of course, a monomial ideal of intersection type, if it  has embedded prime ideals, may have also other primary decompositions than just the canonical one.

\begin{Example}
\label{stanleyreisner}
{\em Let $J$ be the Stanley-Reisner ideal associated to the natural triangulation of the real projective plane, that is,
\[
J = (xyz,xyt,xzu,xtv,xuv,yzv,ytu,yuv,ztu,ztv)\subset S=K[x,y,z,t,u,v].
\]
We apply our Theorem~\ref{prime_intersection} to show that the ideal $J^2$ is not of intersection type. For this, note that $\mm\in \Ass(S/J^2)$ with $xyztuv+J^2$ a non-zero socle element of $S/J^2$ of degree $6$. Since $J^2$ is generated in degree $6$,  Theorem~\ref{prime_intersection}(a) implies that $J^2$ is not of intersection type.}

%Alternatively, one can check by computer that $J^2$ is not of intersection type. Indeed, by using Singular \cite{Si} one obtains  that $J^2=\wt{J^2}\cap %(x^2,y^2,z^2,t^2,u^2,v^2)$ is an irredundant primary decomposition of $J^2$. If we could replace the $\mm$-primary ideal $L=(x^2,y^2,z^2,t^2,u^2,v^2)$ by  a %power of $\mm$ then this power would have to be $\mm^6$, since $\wt{J^2}\subset\mm^5$ and $J^2\not\subset\mm^7$, as can again be checked by using Singular.  That %$L$ cannot be replaced by $\mm^6$ follows again by computer calculations.}
\end{Example}

If the monomial ideal $I$ is of intersection type, then the powers of the prime ideals in the intersection are naturally bounded. More precisely we have

\begin{Theorem}\label{bounded}
Let $I$ be a monomial ideal of intersection type with presentation
$
I=\Sect_{\pp\in \Ass(S/I)}\pp^{d_\pp}.
$
Then the following statements hold:
\begin{enumerate}
\item[(a)] $d_\pp\leq \reg(I(\pp))$ for all $\pp\in \Ass(S/I)$;
\item[(b)] $d_\pp= \reg(I(\pp))$ if and only if $I(\pp)$ has a linear resolution.
\end{enumerate}
In particular, $I=\Sect_{\pp\in \Ass(S/I)}\pp^{\reg(I(\pp))}$  if and only if $I(\pp)$ has linear resolution for all $\pp\in \Ass(S/I)$.
\end{Theorem}

\begin{proof} (a) We consider the following short exact sequence
\begin{eqnarray}
\label{sequence}
0\To \wt{I(\pp)}/I(\pp) \To S(\pp)/I(\pp) \To S(\pp)/\wt{I(\pp)}\To 0.
\end{eqnarray}
Since $\wt{I(\pp)}/I(\pp)$ is a finite length module,  \cite[Corollary 20.19]{Ei} combined with Theorem~\ref{prime_intersection}  implies
\begin{eqnarray}
\label{last}
d_\pp&=&\max(\Soc(\wt{I(\pp)}/I(\pp)))+1=\reg(\wt{I(\pp)}/I(\pp))+1\\
&\leq & \reg(S(\pp)/I(\pp))+1=\reg(I(\pp)). \nonumber
\end{eqnarray}

(b) By Theorem~\ref{prime_intersection}  we have $I(\pp)=\wt{I(\pp)}_{\geq d_\pp}$. Thus if we assume that $I(\pp)$ has a linear resolution it follows that
$\reg(I(\pp))=\min(I(\pp))=d_\pp$.

Conversely, assume that $\reg(I(\pp))=d_\pp$. Then  $I(\pp)=\wt{I(\pp)}_{\geq \reg(I(\pp))}$. Thus by \cite[Theorem 1.2]{EG} it is enough to show that
$\reg(\wt{I(\pp)})\leq \reg(I(\pp))$. But this follows again from \cite[Corollary 20.19]{Ei} according to which
\[
\reg(S(\pp)/I(\pp))=\max\{\reg(\wt{I(\pp)}/I(\pp)),\reg(S(\pp)/\reg(\wt{I(\pp)})\}\geq \reg(S(\pp)/\reg(\wt{I(\pp)}).
\]
\end{proof}

A monomial ideal satisfying the equivalent conditions of Theorem~\ref{bounded} is said to be of {\em strong intersection type}. Obviously, any squarefree monomial ideal is of strong intersection type.

\begin{Example}
{\em  The ideal $$I=(x,y)\cap (x,z)\cap (x,t)\cap(x,y,z,t)^2=(x^2,xy,xz,xt,yzt)$$
of intersection type, and  $\mm=(x,y,z,t)\in \Ass(S/I)$.  But $I$ is not generated in a single degree, let alone  has a linear resolution. Therefore, $I$ is not of strong intersection type. }
\end{Example}

\section{Classes of monomial ideals of (strong) intersection type}

In this section we consider classes of ideals which are of intersection type or strong intersection type.

\medskip
\noindent
{\em  Polymatroidal ideals.} We fix a field $K$, and let $\mathcal P$ be a discrete polymatroid on the ground set $[n]$ of rank $d$, see \cite{HH2} where  these concepts are explained. The polymatroidal ideal associated with $\mathcal P$ is the monomial ideal $I\subset K[x_1,\ldots,x_n]$ generated by all monomials $\xb^\ub$ where $\ub\in B({\mathcal P})\subset \NN^n$. Here $B({\mathcal P})$ denotes the set of bases of $\MP$, namely  the set of all $\ub\in \MP$  whose modulus  $|\ub|=\sum_i u_i$ is maximal. All bases $\ub\in B(\MP)$ have the same modulus, namely $d$, which is defined to be the rank of $\MP$. In particular, $I$ is generated in the single degree $d$.

\begin{Proposition}\label{polym_c}
Any polymatroidal ideal is of strong intersection  type.
\end{Proposition}

\begin{proof}
We recall that a polymatroidal ideal has linear quotients \cite[Theorem 5.2]{CH} and thus has a linear resolution. Moreover, $I(\pp)$ is polymatroidal for any monomial prime ideal $\pp$ which contains $I$ (see \cite[Corollary 3.2]{HRV}). In particular, $I(\pp)$ has a linear resolution for all   $\pp \in \Ass(S/I)$. Thus Theorem~\ref{bounded} yields the desired conclusion.
\end{proof}

In \cite{BaHe} it is conjectured that a monomial ideal is polymatroidal, if all its monomial localizations have a linear resolution. The following example shows that  it does not suffice to require  that $I$ is of strong intersection type to conclude that $I$ is polymatroidal, that is, to require that $I(\pp)$ has a linear resolution for all $\pp\in\Ass(S/I)$.

\begin{Example}
\label{thirdpower} {\em
 Let $J$ be the monomial ideal of  Example~\ref{stanleyreisner}, and let $\pp\in V^*(J^3)$. We claim that $J^3(\pp)$ has a linear resolution if and only if $\pp\in \Ass(S/J^3)$. Calculations with Singular \cite{Si} show that $\Ass(S/J^3)$ consists of all the  prime ideals of height $3,5$ and $6$ which belong to $V^*(J^3)$, altogether these are $17$ prime ideals. Moreover we have the following irredundant primary decomposition of $J^3$
\[
J^3=(v,u,y)^3\cap (v,u,x)^3\cap (v,t,z)^3\cap (v,t,x)^3\cap (v,z,y)^3\cap (u,t,z)^3\cap (t,u,y)^3\cap (u,z,x)^3
\]
\[
\cap (t,x,y)^3\cap (z,x,y)^3\cap (x,y,z,t,v)^6\cap (x,y,z,t,u)^6\cap (x,y,z,v,u)^6\cap
\]
\[
(x,y,t,v,u)^6\cap (x,z,t,v,u)^6\cap (y,z,t,v,u)^6\cap (x,y,z,t,v,u)^9.
\]
We notice that $\mm\in\Ass(S/J^3)$, and by \cite[Corollary 3.3]{Bo} we know that $J^3(=J^3(\mm))$ has a linear resolution. Next consider any prime ideal $\pp\in V^*(J^3)$ of height $5$. Then $J^3(\pp)$ is just the third power of the edge ideal of the cycle of length $5$. For example, if $\pp=(x,y,z,t,u)$ then $$J^3(\pp)=J(\pp)^3=(xt,tz,zy,yu,ux)^3.$$ By Singular one can check that also in this case $J^3(\pp)$ has linear resolution. Finally let $\pp\in V^*(J^3)$ be  of height $3$. Then $\pp$ is a minimal prime ideal of $J^3$. Hence  $J^3(\pp)=\mm_{\pp}^3$, and thus  has a linear resolution. We conclude that  $J^3(\pp)$ has a linear resolution for any associated prime ideal $\pp\in\Ass(S/J^3)$, and hence  $J^3$ is an ideal of strong intersection type.

On the other hand, if $J^3$ would be a polymatroidal ideal, then $J^3(\pp)$ would be polymatroidal for any monomial prime ideal $\pp\in V^*(J^3)$, see \cite[Corollary 3.2]{HRV}. Any height $4$ monomial prime ideal $\pp$ belongs to $V^*(J^3)$ and contains exactly two minimal prime ideals associated to $J^3$. It follows that $J^3(\pp)$ is no longer generated in a single degree and thus can not be polymatroidal. To exemplify this, let $\pp=(x,y,z,t)$. Then $J^3(\pp)=(t,x,y)^3\cap (z,x,y)^3$ is minimally generated in degrees $3,4,5$ and $6$.

Our discussion showed that $\Ass(S/J^3)$ is a proper subset of $V^*(J^3)$, and  that $J^3(\pp)$ has a linear resolution if and only  $\pp\in\Ass(S/J^3)$. }
\end{Example}

Next we want to describe the canonical primary decomposition of a polymatroidal ideal.  Let $\mathcal P$ be a discrete polymatroid on the ground set $[n]$, and denote by $2^{[n]}$ the set of all subsets of $[n]$. For a vector $\ub\in \ZZ^n$ and  $F\in
2^{[n]}$ we set $\ub(F)=\sum_{i\in F}u_i$.

The ground set {\em rank function} of ${\mathcal  P}$ is the function $\rho\: 2^{[n]}\to \ZZ_+$  defined by
setting
\[
\rho(F) = \max\{\ub(F)\:\; \ub\in B({\mathcal  P})\}\quad \text{for all $F\in 2^{[n]}$, $F\neq \emptyset$, }
\]
together with $\rho(\emptyset) = 0$.

\begin{Lemma}
\label{rank}
Let $\mathcal P$ be a discrete polymatroid of rank $d$ on the ground set $[n]$ and $I$ the polymatroidal ideal associated with $\mathcal P$. Furthermore, let $F\subset [n]$ and $\pp_F$ be the monomial prime ideal generated by the variables $x_i$ with $i\in F$. Then $I(\pp_F)$ is generated in degree $d-\rho([n]\setminus F).$
\end{Lemma}

\begin{proof}
As observed before in the proof of Proposition~\ref{polym_c}, the ideal $I(\pp_F)$ is polymatroidal for all $\pp_F$ which contain $I$. If $I\not\subset \pp_F$, then $I(\pp_F)=S(\pp_F)$ which by definition we may also consider as a polymatroidal ideal. The ideal $I(\pp_F)$ is generated in a single degree, say $t_F$. By the definition of monomial localization we have that  $t_F=\min\{\ub(F)\:\; \ub\in B(\MP)\}$. Since
\[
\min\{\ub(F)\:\; \ub\in B(\MP)\}=d-\max\{\ub([n]\setminus F)\:\; \ub\in B(\MP)\},
\]
the  desired formula follows.
\end{proof}

\medskip
The canonical primary decomposition of a polymatroidal ideal is now given as follows:

\begin{Theorem}\label{primdec_polym}
Let $\mathcal P$ be a  discrete polymatroid of rank $d$ with rank function  $\rho$, and let $I$ be the  polymatroidal ideal associated with  $\mathcal P$.  Then
\[
I=\Sect_{\pp_F\in\Ass(S/I)}
\pp_F^{d-\rho([n]\setminus F)}.
\]
is the canonical  primary decomposition of $I$.
\end{Theorem}

\begin{proof} Let $\pp_F\in \Ass(S/I)$. Then  $\reg I(\pp_F)$ is equal to the common degree of the generators of $I(\pp_F)$, since $I(\pp_F)$ has a linear resolution. By Lemma~\ref{rank}, this common degree is $d-\rho([n]\setminus F)$. Since $I$ is of strong intersection type, the desired conclusion follows from Theorem~\ref{bounded}.
\end{proof}

In general it is not so easy to identify the associated prime ideals of a polymatroidal ideal. The complete answer is known for ideals of Veronese type as well as for polymatroidal ideals associated with transversal polymatroids. We describe these cases in the following examples.

\begin{Examples}\label{known_polym}
{\em (a) Let $d, a_1,\ldots,a_n$ be positive integers. The ideal $I_{d;a_1,\ldots,a_n}$ generated by all monomials $x_1^{c_1}\cdots x_n^{c_n}$  of degree $d$ with $c_i\leq a_i$ for $i=1,\ldots,n$ is called of {\em Veronese type}. Ideals of Veronese type are polymatroidal.

Let $I= I_{d;a_1,\ldots,a_n}$ be the ideal of Veronese type with the property that $d\geq a_i$ for all $i$ and $a_1,\ldots,a_n\geq 1$. The set of associated prime ideals is described in \cite[Proposition 3.1]{Vl} as follows
\[
\Ass(S/I)=\{\pp_F\:\ \sum_{i=1}^n a_i\geq d-1+|F| \ \text{and} \ \sum_{i\notin F} a_i\leq d-1\}.
\]
Thus  by applying Theorem~\ref{primdec_polym} the canonical  primary decomposition of $I$ is given as
\[
I=\Sect_{\pp_F\in\Ass(S/I)}\pp_F^{d-\sum_{i\notin F}a_i},
\]
since the rank function of the corresponding polymatroid $\MP$ is given by $\rho(F)=\min\{d,\sum_{i\in F}a_i\}$ for all $F\subset [n]$.

\medskip
Consider for example  $I=I_{4;3,2,1}\subset S=K[x_1,x_2,x_3]$. Then  $$I=(x_1^3x_2,x_1^3x_3,x_1^2x_2^2,x_1^2x_2x_3,x_1x_2^2x_3),$$ and the above formulas  yield
\[
\Ass(S/I)=\{(x_1),(x_2,x_3),(x_1,x_3),(x_1,x_2),(x_1,x_2,x_3)\},
\]
and  the canonical  primary decomposition
\begin{eqnarray*}
I=(x_1)\cap(x_2,x_3)\cap(x_1,x_3)^2\cap(x_1,x_2)^3\cap(x_1,x_2,x_3)^4.
\end{eqnarray*}

(b) Let $I=\pp_{F_1}\cdots \pp_{F_d}\subset K[x_1,\ldots,x_n]$ be a transversal polymatroidal ideal. Without loss of generality we may assume that $\Union_{i=1}^d F_i=[n]$. The set of associated prime ideals of $I$ is described in \cite[Exercise 3.9]{Ei} as well as in \cite[Theorem 4.7]{HRV}, and an irredundant primary decomposition for $I$ is given in \cite[Corollary 4.10]{HRV}. Both descriptions are given in terms of a graph $G_I$ attached to the ideal $I$, see \cite[Section 4]{HRV}. The graph $G_I$ is defined on the vertex set $[d]$ with  $\{i,j\}$  an edge of  $G_I$  if $F_i\cap F_j\neq\emptyset$. In \cite[Theorem 4.7]{HRV} it is shown  that the associated prime ideals  of $I$ corresponded to the trees of the graph $G_I$ in the following way:
$\pp_F\in\Ass(S/I)$ if and only if there exists a tree $\MT$ such that  $\pp_F=\pp_{\mathcal T}$.

In the following we show that the  irredundant primary decomposition given in terms of the rank function  as described in  Theorem~\ref{primdec_polym} coincides with  the one given in \cite[Corollary 4.10]{HRV}. There the irredundant  primary decomposition of $I$ is given as
\[
I=\Sect_{\pp_F\in\Ass(S/I)}\pp_F^{a_F},
\]
where $a_F$ is the number of vertices of a tree $\MT$ in $G_I$ which is maximal with respect to the property  that $\pp_F=\pp_{\mathcal T}$, where $\pp_{\mathcal T}=\sum_{i\in V(\mathcal T)}\pp_{F_i}$.

On the other hand,  let $\rho$ be the rank function of the transversal polymatroid attached  to $I$.  In \cite[Section 9]{HH2} it is shown that
\[
\rho(F)=|\{i:\ F\cap F_i\neq\emptyset\}| \quad \text{ for all } F\subset [n].
\]
It follows that
\begin{eqnarray}
\label{d}
d=|\{i:\  F_i\subset  F\}|+\rho([n]\setminus F).
\end{eqnarray}
Now let $\MT$ be a tree in $G_I$ which is maximal with respect to the property  that $\pp_F=\pp_{\mathcal T}$. Suppose  we have shown that
\[
|\{i:\  F_i\subset F\}|=|\{i\:\; i\in V(\MT)\}|,
\]
then it will follow from \eqref{d} that $a_F=d-\rho([n]\setminus F)$ for all $F$ which indeed  we want to show.

The inequality $|\{i:\  F_i\subset F\}|\geq |\{i\:\; i\in V(\MT)\}|$ is obvious, since $F_i\subset F$ for all $i\in V(\MT)$. Suppose this inequality is strict. Then there exists $j\not \in V(\MT)$ with $F_j\subset F$. Since $F=\Union_{i\in V(\MT)}F_i$, we then notice that $F_j\sect F_i\neq \emptyset$ for some $i\in V(\MT)$, contradicting  the maximality of  $\MT$.
}
\end{Examples}

%\begin{Lemma}
%\label{irredundant}
%Let $I$ be a monomial ideal with presentation $I=\Sect_{i=1}^k\pp_i^{d_i}$,  where the $\pp_i$ are monomial prime ideals. Then there exists $\{i_1,\ldots, i_t\}\subset [k]$ such that $\Ass(S/I)=\{\pp_{i_1},\ldots,\pp_{i_t}\}$ and $I=\Sect_{j=1}^t\pp_{i_j}^{d_{i_j}}$.
%\end{Lemma}

%\begin{proof}
%If the given presentation of $I$ is irredundant, then $\Ass(S/I)=\{\pp_{1},\ldots,\pp_{i_t}\}$. Assume now that the presentation is redundant. Then there exists %an index $j$ such that $I=\Sect_{i=1, j\neq j}^k\pp_i^{d_i}$. Applying induction on $k$ it follows that $\Ass(S/I)\subset \{\pp_i\:\; i=1,\ldots,k,i\neq j\}$.
%\end{proof}

A redundant primary decomposition of a polymatroidal ideal, which in some cases is even irredundant and which is given only in terms of the rank function of the corresponding polymatroid can be described  as follows:

\medskip
Let $\MP$ be a discrete polymatroid of rank $d$ on the ground set $[n]$ and $\rho\: 2^{[n]}\to \ZZ_+$ its rank function. Then $\rho$ satisfies the following conditions:
\begin{enumerate}
\item[(i)] $\rho(\emptyset)=0$;
\item[(ii)] $\rho(F)\leq \rho(G)$ for all $F\subset G\subset [n]$;
\item[(iii)] $\rho(F)+\rho(G)\geq \rho(F\sect G)+ \rho(F\union G)$ for all $F,G\in 2^{[n]}$.
\end{enumerate}

We define the {\em complementary rank function}  $\tau\: 2^{[n]}\to \ZZ_+$ by setting $\tau(F)=d-\rho([n]\setminus F)$ for all $F\in 2^{[n]}$. The function $\tau$ also satisfies the properties (i) and (ii), but instead of (iii) one has:
\begin{enumerate}
\item[(iii')] $\tau(F)+\tau(G)\leq \tau(F\sect G)+ \tau(F\union G)$ for all $F,G\in 2^{[n]}$.
\end{enumerate}

A subset $F\subset [n]$ will be called {\em  $\tau$-closed}, if $\tau (G)<\tau(F)$ for any proper subset $G$ of $F$, and $F$ will be called {\em $\tau$-separable} if there exist non-empty subsets $G$ and $H$ of $F$ with $G\sect H=\emptyset$ and $G\union H=F$ such that $\tau(G)+\tau(H)=\tau(F)$. If $F$ is not $\tau$-separable, then it is called $\tau$-inseparable. Corresponding concepts exist for $\rho$.

\begin{Theorem}
\label{redundant}
Let $I$ be a polymatroidal ideal associated with the discrete polymatroid $\mathcal P$ with complementary rank function $\tau$. Then
\[
I=\Sect_{F}\pp_F^{\tau(F)},
\]
where the intersection is taken over all $F\subset [n]$ which are $\tau$-closed and $\tau$-inseparable.
\end{Theorem}

\begin{proof}
We first show that
$I=\Sect_{F\subset [n]}\pp_F^{\tau(F)},$
where we take the intersection over all subsets $F$ of $[n]$.  Indeed, due to Theorem~\ref{primdec_polym} it suffices to show that $I\subset \pp_F^{\tau(F)}$ for all $F\subset [n]$. In order to show this let $\xb^{\ub}\in G(I)$. Then  $\ub([n])=d$ and $\ub(F)\leq\rho(F)$ for all subsets $F\subset [n]$. Therefore,  $d-\ub(F)=\ub([n]\setminus F)\leq\rho([n]\setminus F)$ and this  implies that $\ub(F)\geq \tau(F)$. In other words, $\xb^{\ub}\in \pp_F^{\tau(F)}$.

In the next step we show that whenever $F$ is not  $\tau$-closed or is  $\tau$-separable, then  $\pp_F^{\tau(F)}$ may be omitted in the intersection $\Sect_{F\subset [n]}\pp_F^{\tau(F)}$ without changing it. Indeed, if $F$ is not  $\tau$-closed, then there exists a proper subset $G$ of $F$ with $\tau(G)=\tau(F)$. It follows that $\pp_G^{\tau(G)}\subset \pp_F^{\tau(F)}$, and hence $\pp_F^{\tau(F)}$ may be omitted. On the other hand, if $F$ is separable, then there exist non-empty subsets  $G$ and $H$ of $F$ with $G\sect H=\emptyset$ and $G\union H=F$ such that $\tau(G)+\tau(H)=\tau(F)$. Thus we see that
\[
\pp_G^{\tau(G)}\sect \pp_H^{\tau(H)}= \pp_G^{\tau(G)} \pp_H^{\tau(H)}\subset (\pp_G+\pp_H)^{\tau(G)+\tau(H)}=\pp_F^{\tau(F)},
\]
and hence again $\pp_F^{\tau(F)}$ may be omitted in the intersection.
\end{proof}

\begin{Example}
{\em The intersection given in Theorem~\ref{redundant} may be an irredundant primary decomposition but also  very far of being irredundant. Consider for example the Veronese type ideal $I_{4;3,2,1}$ analyzed in Example~\ref{known_polym}(a). The rank function $\rho$ of the associated discrete polymatroid is given as follows: $\rho(\emptyset)=0$, $\rho(\{1\})=3$, $\rho(\{2\})=2$, $\rho(\{3\})=1$, $\rho(\{1,2\})=4$, $\rho(\{1,3\})=4$, $\rho(\{2,3\})=3$ and $\rho(\{1,2,3\})=4$. Thus the complementary rank function is the following: $\tau(\emptyset)=0$, $\tau(\{1\})=1$, $\tau(\{2\})=0$, $\tau(\{3\})=0$, $\tau(\{1,2\})=3$, $\tau(\{1,3\})=2$, $\tau(\{2,3\})=1$ and $\tau(\{1,2,3\})=4$. One can easily check that there are no $\tau$-inseparable subsets of $[3]$ and the $\tau$-closed subsets of $[3]$ are: $\{1\}, \{1,2\}, \{1,3\}, \{2,3\},\{1,2,3\}$. Applying now Theorem~\ref{redundant} we obtain that
\[
I_{4;3,2,1}=(x_1)\cap(x_2,x_3)\cap(x_1,x_3)^2\cap(x_1,x_2)^3\cap(x_1,x_2,x_3)^4.
\]
This is the canonical primary decomposition of $I_{4;3,2,1}$ presented already in Example~\ref{known_polym}(a). Thus in this case Theorem~\ref{redundant} yields an irredundant primary decomposition.

On the other hand if we consider the squarefree Veronese ideal $I=I_{n-1;1,\ldots,1}\subset K[x_1,\ldots,x_n]$ we will see that the intersection given in Theorem~\ref{redundant} is very far of being an irredundant primary decomposition. Indeed, the canonical primary decomposition of $I$ is
\[
I=\Sect_{1\leq i<j\leq n} (x_i,x_j).
\]
The rank function $\rho$ of the associated discrete polymatroid of $I$ can be easily computed:
\[
\rho(F)=|F| \quad \text{ for any } F\subsetneq [n],
\]
and $\rho([n])=n-1$. Thus we obtain that $\tau(\emptyset)=0$ and $\tau(F)=|F|-1$ for any nonempty subset $F\subset [n]$. This implies that every set $F\subset [n]$ with $|F|\geq 2$ is $\tau$-closed and thus the intersection given by Theorem~\ref{redundant} is the following
\[
I=\Sect_{i=2}^n (\Sect_{F\subset [n], |F|=i}\pp_F^{i-1}).
\]
}
\end{Example}

\medskip
\noindent
{\em Ideals of Borel type.}  We recall  (see \cite[Definition 2.1]{HPV}) that an {\em ideal of Borel type} is a monomial ideal $I\subset S$ such that
\[
I:x_j^{\infty}=I:(x_1,\ldots,x_j)^\infty \quad \text{ for all } j=1,\ldots,n.
\]
For such an ideal is known that the set $\Ass(S/I)$  is totally ordered. More precisely, if $\pp\in\Ass(S/I)$ then $\pp=(x_1,\ldots,x_j)$ for some $j=1,\ldots,n$,  see \cite[Proposition~4.2.9]{HH}.

The {\em principal Borel ideal} generated by the monomial $u$ is the smallest Borel ideal containing $u$, and is denoted $\langle u\rangle$.

\begin{Proposition}\label{borel_c}
Let $I$ be an ideal of Borel type. The following conditions are equivalent:
\begin{enumerate}
\item[(a)] $I$ is of strong intersection type;
\item[(b)] $I$ is of intersection type;
\item[(c)] $I$ is principal Borel.
\end{enumerate}
\end{Proposition}
\begin{proof}
(a) \implies (b) is obvious. For (b) \implies (c) assume that $I=\Sect_{i=1}^r \pp_i^{d_i}$ is an irredundant primary decomposition of $I$. Since $I$ is of Borel type it follows from the previous comments that we may assume that $\pp_1\subset\pp_2\subset\cdots\subset\pp_r$. Therefore we must have $d_1<d_2<\cdots<d_r$, otherwise the given primary decomposition would not be irredundant. We claim that $I$ is the principal Borel ideal generated by the monomial $$u:=x_{n_1}^{d_1}x_{n_2}^{d_2-d_1}\cdots x_{n_s}^{d_s-d_{s-1}},$$
where for all $i=1,\ldots,r$, $x_{n_i}$ is the variable of highest index appearing in the minimal system of generators of $\pp_i$. Indeed, note first that  $\langle u\rangle =\pp_1^{d_1}\pp_2^{d_2-d_1}\cdots \pp_s^{d_s-d_{s-1}}$. Thus the  ideal $\langle u\rangle$ is a transversal polymatroidal ideal and applying either \cite[Corollary 4.10]{HRV} or \cite[Theorem 4.3]{FMS} we obtain that $\langle u\rangle =\Sect_{i=1}^r \pp_i^{d_i}$ is the irredundant primary decomposition of $\langle u \rangle$.  Therefore,  $I=\langle u \rangle$. In order to prove the implication (c) \implies (a) we use the fact that a principal Borel ideal is polymatroidal and by Proposition~\ref{polym_c} we obtain the desired conclusion.
\end{proof}

\begin{Remark}
{\em (a) Every monomial ideal of the form $I=\Sect_{i=1}^r \pp_i^{a_i}$ with the property that $\pp_i\not\subset \pp_j$ for all $i\neq j$ is of strong intersection type. In particular all squarefree monomial ideals are of strong intersection type, as we already noticed before.

(b) Let $\pp_1\subset \pp_2\subset\cdots\subset \pp_s$ be a chain of monomial prime ideals in $S$ and $d_1<d_2<\cdots<d_s$ be a sequence of positive integers. Then the ideal $I=\Sect_{i=1}^r \pp_i^{d_i}$ is a principal Borel ideal after a relabeling of the variables and by Proposition~\ref{borel_c}, the ideal  $I$ is of strong intersection type.

%(c) The class $\mathcal C$ is not closed with respect to taking powers. To see this let $I$ be the Stanley-Reisner ideal that corresponds to the natural triangulation of the real projective plane. Since $I$ is a squarefree monomial ideal it belongs to $\mathcal C$. However $\mm\in\Ass(I^2)$ and $I^2=I^2(\mm)$ does not have a linear resolution, thus $I^2\notin\mathcal C$.
}
\end{Remark}

%{\bf Question:} Can we say everything about $P_1^a\cap P_2^b$ (regularity, etc..). What about the stable ideals? (do they have an irredundant presentation of the form from Proposition 1.1?)

\medskip
\noindent
{\em Powers of edge ideals.} Let $G$ be a simple graph on the vertex set $[n]$, that is,  a graph without loops and multiple edges. The edge ideal associated with $G$ is the monomial ideal $I(G)\subset K[x_1,\ldots,x_n]$ generated by the monomials $x_ix_j$ corresponding to the edges $\{i,j\}$ of $G$. Here we study the question under which conditions on $G$, powers of $I(G)$ are of (strong) intersection type. We recall that edge ideals have the persistence property, that is,  $\Ass(S/I(G)^k)\subset\Ass(S/I(G)^{k+1})$ for all $k\geq 1$. This is shown  in \cite[Theorem 2.15]{MMV}.

\medskip
We begin with some examples.

\begin{Example}{\em
(1) Let $G$ be a graph such that $\Ass(S/I(G))=\Ass(S/I(G)^k)$ for all $k$. Then $I(G)^k$ is of strong intersection type. This situation is given for example in the case that  $G$ is bipartite,  as shown in \cite[Theorem 5.9]{SVV}.

(2) Let $G$ be a graph with the property that $I(G)^k$ has a linear resolution for all $k\geq 2$ and such that $\Ass(S/I^k) \subset \Ass(S/I)\union \{\mm\}$ for all $k$. Then Theorem~\ref{bounded} implies that $I(G)^k$ is of strong intersection type for all $k$.
The 5-cycle  $C_5$ satisfies this condition. Indeed,  it follows from \cite[Theorem 6.12]{B} that $I(C_5)^k$ has a linear resolution for all $k\geq 2$, and from
\cite[Lemma 3.1]{CMS} that $\Ass(S/I(C_5)^k) \subset \Ass(S/I(C_5))\union \{\mm\}$ for all $k$.

(3) Let $G=C_{2k+1}$ be the odd cycle of length $2k+1$ with $k\geq 3$. Since $G^c$ contains an induced $4$-cycle, it follows that $I(G)^s$ has no linear resolution for any $s\geq 1$, see \cite[Proposition 1.8]{NP}. On the other hand it is  known that $\Ass(S/I(G)^s)=\Ass(S/I(G))$ for $s\leq k$ and $\Ass(S/I(G)^{s})=\Ass(S/I(G))\cup\{\mm\}$ for $s>k$, see \cite[Lemma 3.1]{CMS}. It follows that $I(G)^s$ is of strong intersection type for $s\leq k$, and not of strong intersection type for $s>k$. We do not know whether or not $I(G)^s$ is of intersection type for all $s\geq 1$.

}
\end{Example}

\medskip
Let $G$ be a graph. We call a $3$-cycle $C$ of $G$ {\em central} if all vertices of $G$ are neighbors of $C$.

A slight generalization of the \cite[Theorem 2.1]{HH3} is the following lemma, which we shall need in the proof of the next theorem.

\begin{Lemma}\label{depth_2}
Let $G$ be a graph on the vertex set $[n]$  and let $$I\subset R=K[x_1,\ldots,x_n,y_1,\ldots,y_k]$$ be the monomial ideal $I=(y_1,\ldots,y_k,I(G))$, where $I(G)\subset K[x_1,\ldots,x_n]$. The following conditions are equivalent:
\begin{enumerate}
\item[(a)] $\depth(R/I^2)=0$,
\item[(b)] $G$ is a connected graph containing a central $3$-cycle.
\end{enumerate}
If the equivalent conditions hold and $C$ is a central $3$-cycle of $G$ with $V(C)=\{i,j,k\}$, then $x_ix_jx_k+I^2$ is a non-zero socle element of $R/I^2$.
\end{Lemma}

\begin{proof}
(a)$\Rightarrow$(b): Since $\depth(R/I^2)=0$, there exists a monomial $u\notin I^2$ such that $u\cdot(x_1,\ldots,x_n,y_1,\ldots,y_k)\subset I^2$.  By \cite[Corollary 1.2]{HH3} we know that $u$ is a squarefree monomial. We claim that $u$ is not divisible by any $y_j$. Assuming the claim proved, it follows
that $u\notin I(G)^2$ and $u\cdot(x_1,\ldots,x_n)\subset I(G)^2$. Therefore $\depth(K[x_1,\ldots,x_n]/I(G)^2)=0$, and by \cite[Theorem 2.1.]{HH3} we obtain the desired conclusion.

In order to prove the claim we argue by contradiction. Without loss of generality we may assume that $y_1$ divides $u$. Since $u\notin I^2$ it follows that $u$ can not be further divisible by any $y_i$. Thus $u=x_{i_1}\cdots x_{i_s}y_1$ for some integers $1\leq i_1<\ldots<i_s\leq n$ with $s\geq 1$ because $y_1$ does not belong to the socle of $R/I^2$. Furthermore for every integers $p\neq q$ we have $\{i_p,i_q\}$ is not an edge of $G$, otherwise $u\in y_1I(G)$, a contradiction. This implies that $ux_{i_1}$ is not divisible by any of the minimal generators of $y_1I(G)$ or $I(G)^2$ and therefore $ux_{i_1}\notin I^2$, a contradiction. Hence our claim is proved and we are done.

(b)$\Rightarrow$(a): After a relabeling of the vertices of $G$ we may assume that the $3$-cycle has the edges $\{1,2\},\{1,3\}$ and $\{2,3\}$. It follows from the proof of \cite[Theorem 2.1]{HH3} that the monomial $u=x_1x_2x_3$ satisfies $u\notin I(G)^2$ and $u\in I(G)^2:(x_1,\ldots,x_n)$. It follows that $u\notin I^2$. Moreover, since $uy_i\in y_iI(G)\subset I^2$ for all $i=1,\ldots,k$ it follows that $u\cdot(x_1,\ldots,x_n,y_1,\ldots,y_k)\subset I^2$, and thus $\depth(R/I^2)=0$.
\end{proof}

We want to point out that the proof of Lemma~\ref{depth_2} shows  that if $\depth(R/I^2)=0$ then there exists a socle element of degree $3$.

\begin{Theorem}
\label{depth_2power}
Let $G$ be a graph on the vertex set $[n]$ with no isolated vertices, and let  $I=I(G)$ be the  edge ideal of $G$ in $S= K[x_1,x_2,\ldots,x_n]$. The following conditions are equivalent:
\begin{enumerate}
\item[(a)] $I^2$ is of intersection type,
\item[(b)] $\Ass(S/I^2)\subset \Ass(S/I)\cup\{\mm\}$,
\item[(c)] All $3$-cycles of the graph $G$ are central.
\end{enumerate}
\end{Theorem}

\begin{proof}
We prove the theorem by showing that (a)$\Rightarrow$(b)$\Rightarrow$(c)$\Rightarrow$(a).

(a)$\Rightarrow$(b): Since $\Ass(S/I)=\Min(S/I)\subset\Ass(S/I^k)$ for all $k\geq 1$, it follows that $\Ass(S/I)\subset\Ass(S/I^2)$. Assume now that there exists a monomial prime ideal $\pp\in\Ass(S/I^2)$ which does not belong to $\Ass(S/I)\cup\{\mm\}$. If $I(\pp)$ is only generated by variables, then $I(\pp)=\pp S(\pp)$, because otherwise  $\depth(S(\pp)/I(\pp)^2)>0$. It follows that   $\pp\in\Ass(S/I)$, a contradiction. Therefore, $I(\pp)$ contains also generators of degree $2$, and we may apply Lemma~\ref{depth_2}, and  conclude  that $\max(\Soc(S(\pp)/I(\pp)^2))\geq 3$. On the other hand,  by the assumption on $G$, each variable $x_i$ divides some generator of $I$.  Since at least one variable is mapped to 1 in $S(\pp)$, it follows that  $I(\pp)$ contains also variables, and hence  $\min(I(\pp)^2)=2$.  By  Theorem~\ref{prime_intersection} this contradicts our assumption that $I^2$ is of intersection type.

(b)$\Rightarrow$(c): Suppose that there exists a $3$-cycle $C$ of $G$  which is not central. Let $D$ be the set of vertices of $G$ which are not neighbors of $C$, and let $\pp$ be the monomial prime ideal generated by the variables $x_i$ with $i\not\in  D$.  Then $I(\pp)=I(H)+Q$ where $Q$ is generated by a set $X$ of variables,  and where $I(H)$ is an edge ideal in a set of variables disjoint from $X$ with  $H$ a connected graph containing the $3$-cycle $C$ which is central in $H$. Lemma~\ref{depth_2} implies that $\depth(S(\pp)/I(\pp)^2)=0$. It follows that $\pp\in \Ass(S/I^2)$. The prime ideal $\pp$ is not a  minimal prime ideal of $I$  since it contains monomial generators of $I$ of degree 2, and it is not maximal since $D\neq \emptyset$. This is a contradiction.

(c)$\Rightarrow$(a): Let $\pp\neq \mm$ be a monomial prime ideal containing $I^2$, and write  $I(\pp)=I(H)+Q$ as in the proof of (b)\implies (c).  This time however,  since all $3$-cycles of $G$  are central, it follows that $H$ contains no $3$-cycle at all. Thus from Lemma~\ref{depth_2} one deduces that $\depth(S(\pp)/I(\pp)^2)=0$ if and only if $I(\pp)=\pp S(\pp)$, in which case $\pp$ is a minimal prime ideal. On the other hand, if $\mm\in \Ass(S/I^2)$, then $\depth(S/I^2)=0$. By \cite[Corollary 1.2]{HH3},  $\Soc(S/I^2)$ is generated by elements of the form $u+I^2$ where $u$ is a squarefree monomial. Let $u+I^2$ be such a non-zero socle element, and let $H$ be the induced subgraph of $G$ whose vertex set is the support of $u$. In the proof of \cite[Theorem  2.1, (a)\implies (b)]{HH3} it is shown that $H$ is either a 3-cycle  or a line of length at most 2. It follows that $\deg u \leq 3$. Hence we have shown that $\max(\Soc(S/I^2))\leq 3$. Since $I^2$ is generated in degree $4$, Theorem~\ref{prime_intersection} yields  the desired conclusion.
\end{proof}

\begin{Corollary}
\label{higher}
Let $I$ be an edge ideal such that $I^2$ is not of intersection type. Then  for any $k\geq 2$,  the ideal $I^k$ is not of intersection type.
\end{Corollary}

\begin{proof}
Let $I=I(G)$. Since $I^2$ is not of intersection type, Theorem~\ref{depth_2power} implies that there  exists a $3$-cycle $C$ contained in $G$ which is not central. As in the proof of Theorem~\ref{depth_2power}(b)\implies (c) it follows that there exists a monomial prime ideal $\pp\neq \mm$ such that $I(\pp)=I(H)+Q$ where  $H$ is a connected graph containing the $3$-cycle $C$ which now is central in $H$, and where  $Q$ is generated in a disjoint set of variables, say $I(H)\subset K[x_1,\ldots,x_k]$ and $Q=(x_{k+1},\ldots,x_\ell)$. Furthermore we may assume that $V(C)=\{1,2,3\}$. It follows from Lemma~\ref{depth_2} that $u+I(\pp)^2$ with $u=x_1x_2x_3$ is a non-zero socle element of $S(\pp)/I(\pp)^2$. This implies that $(x_1x_2)^{k-2}u+I(\pp)^{k}\in \Soc(S(\pp)/I(\pp)^k).$

Assume that $(x_1x_2)^{k-2}u\in I(\pp)^k$. Then $(x_1x_2)^{k-2}u\in I(H)^{k-s}Q^s$ for some $s$ with $0\leq s\leq k$. Since $(x_1x_2)^{k-2}u$ is not divisible by any $x_i\in Q$, it follows that $s=0$, so that $(x_1x_2)^{k-2}u\in I(H)^k$. This is a contradiction, since $\deg (x_1x_2)^{k-2}u=2k-1$, while the generators of $I(H)^k$ are of degree $2k$.

We conclude that  $(x_1x_2)^{k-2}u+I(\pp)^{k}$ is a non-zero socle element of $S(\pp)/I(\pp)^k$, and hence $\max(\Soc(S(\pp)/I(\pp)^k)\geq 2k-1$. On the other hand $\min(I(\pp)^k)=k$. Thus Theorem~\ref{prime_intersection} implies that $I^k$ is not of intersection type.
\end{proof}

\begin{Examples}
{\em (1) A simple example to which Corollary~\ref{higher} applies is the following: let $I=(xy,xz,yz,zt,tu)$ be the edge ideal of the graph $G$ consisting of the  $3$-cycle $\{x,y,z\}$ and the edges $\{z,t\}$ and  $\{t,u\}$. The $3$-cycle $\{x,y,z\}$ is not central. Therefore by Theorem~\ref{depth_2power}, $I^2$ is not of intersection type,  and thus  by   Corollary~\ref{higher}, $I^k$ is not of intersection type for all $k\geq 2$.

(2) Let $I=I(G)=(xy,yz,xz,yt,zt,tu)$ be the edge ideal of the graph $G$. Note that $\Ass(S/I^2)=\Ass(S/I)\cup\{(x,y,z,t),\mm\}$ and  that the cycle $\{x,y,z\}$ is not central. Since $G^c$ is chordal, \cite[Theorem 10.2.6]{HH} implies that  $I^k$ has a linear resolution for all $k\geq 1$. On the other hand, it follows from Theorem~\ref{depth_2power} and Corollary~\ref{higher} that $I^k$ is not of intersection type for all $k\geq 2$.}
\end{Examples}

\section{Some properties of monomial ideals of (strong) intersection type}

We cannot expect that ideals of (strong) intersection type  have, compared with arbitrary monomial ideals, much better properties, since for example any squarefree monomial ideal is of strong intersection type. However we have

\begin{Proposition}\label{integrally_closed}
Monomial ideals of intersection type are integrally closed.
\end{Proposition}

\begin{proof}
Let $I=\Sect_{\pp\in \Ass(S/I)}\pp^{d_{\pp}}$, and let $u=x_1^{a_1}x_2^{a_2}\cdots x_n^{a_n}$ be a monomial with $u^t\in I^t$ for some $t>0$. Then $u^t\in \pp^{td_\pp}$ for all $\pp\in \Ass(S/I)$. Thus for $\pp\in \Ass(S/I)$, we see  that $\sum_{i, x_i\in\pp}ta_i\geq td_\pp$. Therefore, $\sum_{i, x_i\in\pp}a_i\geq d_\pp$ which implies that $u\in \pp^{d_\pp}$. It follows that $u\in I$, and this proves that $I$ is integrally closed, see \cite[Theorem 1.4.2]{HH}.
\end{proof}

\begin{Corollary}\label{normality}
If $I^k$ is of intersection type for all $k\geq 1$,  then $I$ is normal.
\end{Corollary}

Not all integrally closed ideals are of intersection type. The simplest such example is $I=(x,y^2)$.

\medskip
For a monomial ideal $I\subset K[x_1,\ldots,x_n]$ we denote by $\con(I)$ the convex hull in $\RR^n$ of the  set $\{\ab\:\, \xb^\ab\in I\}$. The set $\con(I)$ is called the {\em Newton polyhedron} of $I$. It is well-known that $I$ is integrally closed if and only if $I$ is generated by the monomials $\xb^\ab$ with $\ab\in \con(I)$.

For a monomial ideal $I$ of intersection type, the supporting hyperplanes of $\con(I)$ can be easily described.  For a subset $F\subset [n]$ we let, as before,  $\pp_F$ be the monomial prime ideal generated by the variables  $x_i$ with $i\in F$. Associated with $\pp_F^a$ we consider the  hyperplane
\[
H_{F,a}=\{\bold{\xi}\in \RR^n\:\ \sum_{i\in F}\xi_i=a\}.
\]
It is clear that $\xb^{\cb}\in \pp_F^a$ if and only if $\cb$ belongs to the half space
\[
H_{F,a}^+=\{\bold{\xi}\in \RR^n\:\ \sum_{i\in F}\xi_i\geq a\}.
\]

Thus the following corollary follows directly from  the definition of ideals  of intersection type.

\begin{Corollary}
\label{supporting}
Let $I$ be a monomial ideal of  intersection type, and let
$$I=\Sect_{\pp_F\in \Ass(S/I)}\pp_F^{d_F}$$ be the canonical  primary decomposition of $I$. Then the hyperplanes $H_{F,d_F}$ with $\pp_F\in \Ass(S/I)$ together with the hyperplanes $H_i=\{\bold{\xi}\in \RR^n\:\; \xi_i\geq 0\}$ for which $(x_i)\not\in \Ass(S/I)$ are the supporting hyperplanes of $\con(I)$.
\end{Corollary}

\medskip
N.V.~Trung, in a discussion with the first author of this paper,  sketched an argument which shows that for any  squarefree monomial ideal $I\subset S$ one has that $I^{(rk)}\subset I^k$ for all $k\geq 1$. Here $r$ denotes the highest degree of a generator in the (unique) minimal set of monomial generators of $I$, and $I^{(t)}$ denotes the $t^{th}$ symbolic power of $I$ which for a monomial ideal we define as follows: Let $Q_1,\ldots, Q_s$ be the uniquely determined primary components corresponding to the minimal prime  ideals of $I$. Then $I^{(t)}=\Sect_{i=1}^sQ_i^t$. (It is also common to define the $t^{th}$ symbolic power of $I$ to be $S\sect \Sect_{\pp\in Ass(S/I)}I^tS_\pp$. This symbolic power is contained in the one defined here.)  In general we expect that $I^{(rk)}\subset I^k$ for all $k$ and all monomial ideals. On the other hand, it follows from a result of Ein, Lazarsfeld and Smith \cite[Theorem A]{ELS}, that for all $k$,  $J^{(ck)} \subset  J^k$ for any graded ideal $J \subset S$ of codimension $c$. Of course there are cases where $r>c$ and vice versa. In the squarefree case however both numbers are bounded by the dimension  of $S$.

Here we prove a related result for monomial ideals of intersection type.  Recall that an ideal $J\subset I$  is said to be a {\em reduction} of $I$,  if $I^k = JI^{k-1}$ for some $k\geq 2$. Following Kodiyalam we denote by $\rho(I)$ the minimum of $\theta(J)$  over all graded
reductions $J$ of $I$, where $\theta(J)$ denotes the maximal degree of a generator in a minimal set of generators of $J$.

\begin{Theorem}\label{symbolic_power}
Let $I\subset S$ be a monomial ideal. Then the following hold:
\begin{enumerate}
\item[(a)] if $I^k$ is of intersection type, then $I^{(\reg(I^k))}\subset I^k$;
\item[(b)] if $I^k$ is of intersection type for all $k\gg 0$, then $I^{(rk)}\subset I^k$ for all $k\gg 0$,
where $r=\rho(I)+1$. In particular, $r\leq \max(I)+1$. Moreover, if $I^k$ has a linear resolution for all $k\gg  0$, then $r$ can be chosen to be $\max(I)$.
\end{enumerate}
\end{Theorem}
\begin{proof}
(a) Set $s=\reg(I^k)$, and let $\qq\in \Ass(S/I^k)$.  There exists $\pp\in \Min(S/I)$ with $\pp\subset \qq$. It follows that $\pp^s\subset \qq^s\subset \qq^{d_{\qq}}$, because $d_\qq\leq \reg(I^k(\qq))\leq s$, see Theorem~\ref{bounded}. This implies that
\[
I^{(s)}=\Sect_{\pp\in \Min(S/I)}\pp^s\subset \Sect_{\pp\in \Ass(S/I^k)}\pp^{d_\pp}=I^k.
\]

(b) By  Kodiyalam \cite{K} we know that $\reg(I^k)=\theta(I)k+c$ for $k\gg 0$. If $I^k$ has a linear resolution for all $k\gg 0$, then $c=0$, and the assertion follows from (a). Otherwise $c>0 $, and for all $k\geq c$ one has that $(\theta(I)+1)k\geq \reg(I^k)$. Again (a) yields the desired conclusion.
\end{proof}

\begin{Corollary}
\label{symbolic_poly}
Let $I$ be a polymatroidal ideal, generated in degree $d$. Then $I^{(dk)}\subset I^k$ for all $k$.
\end{Corollary}


\begin{thebibliography}{10}



\bibitem{BaHe} S.\ Bandari, J.\ Herzog, Monomial localizations and polymatroidal ideals, Eur. J. Combinatorics {\bf 34}, 752--763 (2013).

\bibitem{B} A.\ Banerjee, Bounds on the regularity of ideals associated to simple graphs, preprint 2013.

\bibitem{Bo} K.\ Borna, On linear resolution of powers of an ideal, Osaka J. Math. {\bf 46}, 1047--1058 (2009).

\bibitem{BH} W.\ Bruns, J.\ Herzog, Cohen--Macaulay rings, Revised Edition, Cambridge University Press, Cambridge, 1996.

%\bibitem{C} CoCoATeam, CoCoA: a system for doing Computations in Commutative Algebra, available at http://cocoa.dima.unige.it.

\bibitem{CMS} J.\ Chen, S.\ Morey, A.\ Sung, The stable set of associated primes of the ideal of a graph, Rocky Mountain J. Math. {\bf 32}, 71--89 (2002).

\bibitem{CH} A.\ Conca, J.\ Herzog, Castelnuovo-Mumford regularity of products of ideals, Collectanea Math. {\bf 54}, 137--152 (2003).

\bibitem{ELS} L.\ Ein, R.\ Lazarsfeld, K.\ Smith, Uniform bounds and symbolic powers on smooth varieties, Inventiones Math. {\bf 144}, 241--252 (2001).

\bibitem{Ei} D.\ Eisenbud, Commutative Algebra with a View Toward Algebraic Geometry, Springer, 1994.

\bibitem{EG} D.\ Eisenbud and S.\ Goto, Linear Free Resolutions and Minimal Multiplicity, J. Algebra {\bf 88}, 89--133 (1984).

%\bibitem{F} C.\ Francisco, Resolutions of small sets of fat points, J. Pure Appl. Algebra {\bf 203} no. 1--3, 220--236 (2005).

\bibitem{FMS} C.\ Francisco, J.\ Mermin and J.\ Schweig, Generalizing the Borel property, J. London Math. Soc. (2) {\bf 87}, 724--740 (2013).

\bibitem{FMN} C.\ Francisco, J.\ Migliore and U.\ Nagel, On the componentwise linearity and the minimal free resolution of a tetrahedral curve, J. Algebra {\bf 299}, 535--569 (2006).

\bibitem{FV} C.\ Francisco and A.\ Van Tuyl, Some families of componentwise linear monomial ideals, Nagoya Math. J. {\bf 187}, 115--156 (2007).

\bibitem{Si} G.\ M.\ Greuel, G.\ Pfister and H. Sch\"onemann, Singular 2.0. A Computer Algebra System for Polynomial Computations. Centre for Computer Algebra, University of Kaiserslautern (2001), http://www.singular.uni-kl.de.

\bibitem{HH} J.\ Herzog, T.\ Hibi, {\em Monomial Ideals}. GTM 260. Springer 2010.

\bibitem{HH2} J.\ Herzog, T.\ Hibi, Discrete polymatroids, J. Algebraic Combinatorics {\bf 16}, 239--268 (2002).

\bibitem{HH3} J.\ Herzog, T.\ Hibi, Bounding the socles of powers of squarefree monomial ideals, arXiv:1308.5400v1.

\bibitem{HPV} J.\ Herzog, D.\ Popescu, M.\ Vladoiu, On the Ext-modules of ideals of Borel type, Contemporary Mathematics {\bf 331}, 171--186 (2003).

\bibitem{HRV} J.\ Herzog, A.\ Rauf, M.\ Vladoiu, The stable set of associated prime ideals of a polymatroidal ideal, J. Algebraic Combinatorics {\bf 37}, 289--312 (2013).

\bibitem{K} V.\ Kodiyalam, Asymptotic behaviour of Castelnuovo--Mumford regularity, Proc. AMS Vol. {\bf 128},  407--411 (1999).

\bibitem{MMV} J.\ Martinez--Bernal, S.\ Morey, R.\ Villarreal, Associated primes of powers of edge ideals, Collect. Math. {\bf 63}, 361--374 (2012).

\bibitem{MN} J.\ Migliore and U.\ Nagel, Tetrahedral curves, Int. Math. Res. Notices {\bf 15}, 899--939 (2005).

\bibitem{NP} E.\ Nevo, I.\ Peeva, $C_4$-free edge ideals, J. Algebraic Combinatorics {\bf 37}, 243--248 (2013).

\bibitem{SVV} A.\ Simis, W.\ Vasconcelos, R.\ Villarreal, On the ideal theory of graphs, J. Algebra {\bf 167}, 389--416 (1994).

\bibitem{Vl} M.\ Vladoiu, Equidimensional and unmixed ideals of Veronese type, Comm. in Algebra {\bf 36}, 3378--3392 (2008).

\end{thebibliography}
\end{document}